\documentclass[11pt]{article}
\usepackage[english]{babel}
\usepackage{amssymb}
\usepackage{amsmath}
\usepackage{amsthm}
\usepackage[a4paper,margin=2.54cm]{geometry}
\usepackage{graphicx}
\usepackage{todonotes}
\usepackage{comment}
\usepackage{appendix}
\usepackage{tikz}
\usetikzlibrary{decorations.pathreplacing}
\usetikzlibrary{matrix}
\usepackage{pgfopts}
\usepackage{array}
\usepackage{ytableau}
\usepackage{mathtools}

\theoremstyle{plain}
\newtheorem*{theorem*}{Theorem}
\newtheorem{theorem}{Theorem}[section] 

\newtheorem{proposition}[theorem]{Proposition}
\newtheorem{corollary}[theorem]{Corollary}
\newtheorem{conjecture}[theorem]{Conjecture}

\theoremstyle{definition}
\newtheorem{definition}[theorem]{Definition}

\makeatletter
\@addtoreset{equation}{section}
\makeatother
\numberwithin{equation}{section}

\DeclareMathOperator{\Wr}{Wr}

\DeclareMathOperator{\He}{\rm He}

\DeclareMathOperator{\htt}{ht}

\DeclareMathOperator{\strict}{\mathbf{n}}
\DeclareMathOperator{\Var}{Var}

\newcommand\covby{\lessdot}
\newcommand\cov{\gtrdot}

\newcommand\rhup[2]{\mathcal{R}^+_{#1}(#2)}
\newcommand\rhdown[2]{\mathcal{R}^-_{#1}(#2)}

\DeclarePairedDelimiter\iprod{\langle}{\rangle}

\usepackage{authblk}
\title{Wronskian Appell Polynomials and Symmetric Functions}

\author[1]{Niels Bonneux}
\author[2]{Zachary Hamaker}
\author[3]{John Stembridge}
\author[4]{Marco Stevens}
\affil[1,4]{KU Leuven, Department of Mathematics, 
	
	E-mail: {\tt niels.bonneux@kuleuven.be} and {\tt marco.stevens@kuleuven.be}
}
\affil[2]{University of Florida, Department of Mathematics, 
	
	E-mail: {\tt zhamaker@ufl.edu}
}
\affil[3]{University of Michigan, Department of Mathematics,

E-mail: {\tt jrs@umich.edu}

}
\date{}                     
\providecommand{\keywords}[1]{\textit{\textit{Keywords: }} #1}

\usepackage{hyperref}
\hypersetup{colorlinks, citecolor=black, linkcolor=black}
\AtBeginDocument{} 

\begin{document}
\maketitle

\begin{abstract}\label{abstract}
We study Wronskians of Appell polynomials indexed by integer partitions.
These families of polynomials appear in rational solutions of certain Painlev\'e equations and in the study of exceptional orthogonal polynomials.
We determine their derivatives, their average and variance with respect to Plancherel measure, and introduce several recurrence relations.
In addition, we prove an integrality conjecture for Wronskian Hermite polynomials previously made by the first and last authors.
Our proofs all exploit strong connections with the theory of symmetric functions.
\end{abstract}


\section{Introduction}
\label{sec:introduction}
Let  $(A_n)_{n=0}^\infty$ be a sequence of Appell polynomials;
i.e., a sequence of univariate polynomials such that $A_0=1$ and
$A'_n=nA_{n-1}$ for $n\ge1$.
In this paper, we study the Wronskians of such polynomials; i.e.,
polynomials of the form
\begin{equation}
\label{eq:directWAP}
\frac{\Wr[A_{n_1},A_{n_2},\dots,A_{n_r}]}{\Delta(\strict)},
\end{equation}
where $\Wr$ denotes the Wronskian operator,
$\strict=(n_1,n_2,\dots,n_r)$ is a vector of distinct non-negative integers, and
\[
\Delta(x_1,x_2,\dots,x_r)
  =\det[x_i^{j-1}]_{1\le i,j\le r}
  \ =\prod_{1\le i<j\le r}(x_j-x_i)
\]
is the Vandermonde determinant. The factor $\Delta(\strict)$ here acts as
a normalizing constant so that the resulting polynomials
are monic \cite[Lemma 2.1]{Bonneux_Stevens}. It is clear that~\eqref{eq:directWAP} is invariant
under permutations of $\strict$, so there is no loss of generality in
assuming that the parameters are strictly increasing and positive.
Thus for each integer partition
$\lambda=(\lambda_1\ge\lambda_2\ge\cdots\ge\lambda_r>0)$, we define
\begin{equation}
\label{eq:WAP2}
A_\lambda:=\frac{\Wr[A_{n_1},A_{n_2},\dots,A_{n_r}]}{\Delta(\strict)},
  \quad\text{where
  $\strict=(\lambda_r,\,\lambda_{r-1}+1,\,\dots,\,\lambda_1+r-1)$},
\end{equation}
and refer to these as \textbf{Wronskian Appell polynomials}. It is not hard to check that if 0 is allowed as a part of $\lambda$, there is no effect on~\eqref{eq:WAP2} if those parts are deleted.

Polynomials of this type (for specific partitions) come into play in the rational solutions of the Painlev\'e equations: for Painlev\'e III \cite{Clarkson-PIII,Kajiwara_Masuda} and Painlev\'e V \cite{Clarkson-PV}, the corresponding Appell polynomials are of a modified Laguerre type, for Painlev\'e IV \cite{Clarkson-PIV,Noumi_Yamada,Okamoto} they are of Hermite type, and for Painlev\'e VI \cite{Mazzocco} they are of modified Jacobi type.
For Painlev\'e II, the rational solutions are in terms of Yablonskii-Vorobiev polynomials, which can be expressed in terms of a Wronskian of certain Appell polynomials \cite{Kajiwara_Ohta}. For an overview of these rational solutions see for example \cite{Clarkson-survey} or \cite{VanAssche} and the references therein.
Moreover, Wronskians of Hermite \cite{Duran-Hermite,GomezUllate_Grandati_Milson}, Laguerre \cite{Bonneux_Kuijlaars, Duran-Laguerre} and Jacobi polynomials \cite{Bonneux,Duran-Jacobi} also occur in the study of exceptional orthogonal polynomials.

In \cite{Bonneux_Stevens}, the first and last authors studied Wronskians of Hermite polynomials, which are used to define exceptional Hermite polynomials and solutions for the Painlev\'e IV equation.
They introduced a new recursive formula for computing these polynomials called the ``generating recurrence'' and used it to show that the average of these polynomials is a monomial with respect to Plancherel measure.

In this paper, we extend all of the results from~\cite{Bonneux_Stevens} to the Wronskian polynomials determined by any Appell sequence $(A_n)_{n=0}^\infty$.
To do this, we construct a homomorphism $\varphi_A$ from the ring of symmetric functions $\Lambda$ to the polynomial ring $\mathbb{R}[x]$ that sends augmented Schur functions to polynomials having the form of \eqref{eq:WAP2}.
All of our results on these polynomials may then be deduced from results about symmetric functions.
The first and last authors have proved our main results (with the exception of Theorems~\ref{thm:secondmoment} and \ref{thm:integercoefficients}) by a direct approach bypassing the theory of symmetric functions, as they did for the Hermite case in \cite{Bonneux_Stevens}. The advantage of the symmetric function approach is that it provides extra structure that would otherwise be invisible at the level of univariate polynomials.

In \cite{Segeev_Veselov}, Sergeev and Veselov introduced ``generalized''
Schur polynomials and used them to construct
families of \textit{multivariate} orthogonal polynomials.
In recent work of Grandati~\cite{Grandati}, one sees that in the
confluent limit $\{x_i \to x\}_{i=1}^n$, the polynomials of Sergeev and
Veselov become Wronskians of univariate orthogonal polynomials,
although not necessarily from an Appell sequence.

The remainder of the article is organized as follows.
In Section \ref{sec:mainresults}, we give a high-level overview of our main results.
In Section \ref{sec:preliminaries}, we provide the necessary background on partitions, symmetric functions and Appell polynomials.
The homomorphism $\varphi_A$ is introduced in Section~\ref{sec:images}, while the main results for Wronskians of Appell polynomials are in the subsequent Sections \ref{sec:corollaries} and~\ref{sec:rimandrecurrence}.
We close the article in Section \ref{sec:specificAppellsequences} by explaining how to interpret our results in terms of Appell sequences that appear in applications, such as Wronskians of Hermite polynomials.

\section{Overview of the main results}
\label{sec:mainresults}

The following results refer to Wronskian Appell polynomials $A_\lambda$
as in~\eqref{eq:WAP2}.

\begin{itemize}
\item In Section \ref{sec:images}, we define a ring homomorphism $\varphi_A$ from symmetric functions to polynomials and show that Wronskian Appell polynomials are the images of ``augmented'' Schur functions (Theorem~\ref{thm:imageschur}).
We also discuss the images of other symmetric functions.
All subsequent results are proved by applying $\varphi_A$ to symmetric function identities.
\item The derivative of the Wronskian Appell polynomial $A_\lambda$
can be expressed in terms of the polynomials $A_\mu$ associated to those
partitions $\mu$ that are \textit{covered} by $\lambda$ in Young's lattice (Theorem~\ref{thm:WAPisAN}).
This relation resembles the Appell property $A_n' = nA_{n-1}$ and generalizes~\cite[Proposition 3.5]{Bonneux_Stevens} from the Hermite case to arbitrary Appell polynomials.
\item We compute the average value (Theorem~\ref{thm:average}) and second moment (Theorem~\ref{thm:secondmoment}) of each Wronskian Appell polynomial with respect to the Plancherel measure. The former generalizes~\cite[Theorem 3.4]{Bonneux_Stevens}.
\item As a consequence of the Murnaghan-Nakayama Rule, we derive a collection
of ``top-down'' relations that express $A_\lambda$ in terms
of higher degree Wronskian Appell polynomials (Theorem~\ref{thm:topdown}).
This generalizes~\cite[Theorem 3.2]{Bonneux_Stevens}.
\item The degree-increasing nature of the previous result makes it unsuitable for use in inductive arguments. 
In Section~\ref{sec:rimandrecurrence}, we prove a Schur function generalization of Newton's identities (Theorem~\ref{thm:symmetricgeneratingrecurrence}) that we have not seen elsewhere in the literature.
As a consequence, we obtain a recurrence that expresses $A_\lambda$ in
terms of lower degree Wronskian Appell polynomials (Theorem~\ref{thm:generatingrecurrence}).
This generalizes the fundamental result of \cite[Theorem 3.1]{Bonneux_Stevens} from which all other results in that paper are derived.
\item Theorem \ref{thm:WAPisAN} implies that the Wronskian Appell polynomials contain two distinguished Appell sequences: one associated with the
partitions $(n)$ (i.e., the initial Appell sequence) and another associated
with the partitions $(1^n)=(1,1,\dots,1)$.
We call the latter the \textit{dual} of the original Appell sequence.
In Section \ref{sec:dualappell} we study some of its properties.
\item In Section \ref{sec:integercoefficients}, we introduce a
condition on Appell sequences that is sufficient to force the associated
Wronskian Appell polynomials to have integer coefficients.
This allows us to deduce that Wronskian Hermite polynomials have integer
coefficients (Corollary~\ref{cor:HeZ}), thereby confirming \cite[Conjecture 3.7]{Bonneux_Stevens}.
\end{itemize}

\section{Preliminaries}
\label{sec:preliminaries}

In this section, we introduce some notation and terminology for
working with integer partitions and symmetric functions.
There are many excellent resources for these topics,
for example~\cite{Baik_Deift_Suidan,MacDonald,Stanley_EC2}.
In Section~\ref{sec:Appelldefs},
we review Appell sequences and Wronskian Appell polynomials.

\subsection{Partitions and Young's lattice}
A non-negative integer sequence
$\lambda = (\lambda_1 \geq \lambda_2 \geq \dots)$ is a \textbf{partition}
if $|\lambda|:=\sum_{i=1}^{\infty} \lambda_i$ is finite.
If $|\lambda|=m$, then $\lambda$ is said to be a partition of $m$
(or of \textbf{size} $m$) and we write $\lambda \vdash m$.
The \textbf{length} of~$\lambda$, denoted $\ell(\lambda)$,
is the largest index $r$ such that $\lambda_r>0$.
We often write $\lambda=(\lambda_1,\lambda_2,\dots, \lambda_r)$.
The unique partition of~0 is denoted $\emptyset$.
The \textbf{diagram} of a partition is
\[
D_\lambda=\{(i,j)\in\mathbb{Z}^2: 1\le i\le\ell(\lambda),
  \ 1\le j\le\lambda_i\}.
\]
The points $(i,j)\in D_\lambda$ are often depicted as unit squares
with matrix-style coordinates. We partially order partitions component-wise,
or equivalently, by inclusion of diagrams, so that
\[
\mu\le\lambda\ \ \text{if $\ D_\mu \subseteq D_\lambda$}.
\]
This partial ordering of partitions is known
as \textbf{Young's lattice} and denoted $\mathbb{Y}$.
It has a unique minimal element $\emptyset$ and is graded by size.

Given a pair $\mu\le\lambda$, the difference
\[
D_{\lambda/\mu}:= D_\lambda \setminus D_\mu
\]
is called a \textbf{skew diagram} of shape $\lambda/\mu$.  For example, 
\begin{equation*}
\label{eq:diagram}
\ytableausetup{smalltableaux}
D_{(2,1)} = \ydiagram{2,1}, \quad
D_{(4,3,2)} = \ydiagram{4,3,2} \quad
\mbox{and} \quad D_{(4,3,2)/(2,1)} = \ydiagram{2+2,1+2,2}\,.
\end{equation*}
The \textbf{conjugate} of $\lambda$, denoted $\lambda'$, is the partition
whose diagram is $\{(i,j):(j,i)\in D_\lambda\}$.
For example, $(2,1)' = (2,1)$ and $(4,3,2)' = (3,3,2,1)$.

We write $\mu\covby\lambda$ or $\lambda\cov\mu$ to indicate that
$\lambda$ covers $\mu$ in $\mathbb{Y}$; i.e., $\mu<\lambda$
and $|\lambda|-|\mu|=1$.
A \textbf{standard Young tableau} of shape $\lambda/\mu$ is a maximal saturated chain from $\mu$ to $\lambda$ in Young's lattice; i.e., a sequence
$\mu = \nu^{(0)} \covby \nu^{(1)} \covby \cdots \covby \nu^{(m)} = \lambda$.
We let $F_{\lambda/\mu}$ denote the number of such tableaux.
In case $\mu=\emptyset$, we identify $\lambda/\mu$ with $\lambda$,
so that $F_{\lambda} = F_{\lambda/\emptyset}$.


For any partition $\lambda$ and $(i,j) \in D_\lambda$, the \textbf{hook length}
at $(i,j)$ is $h(i,j) = \lambda_i - j + \lambda_j' - i +1$.
This counts the number of cells in $D_\lambda$ that are directly
below or directly to the right of $(i,j)$, including $(i,j)$.
These hook lengths occur in the classic hook formula for counting
the  standard Young tableaux of shape $\lambda$; namely,
\begin{equation}
\label{eq:Flambdahooklengths}
F_\lambda=\frac{|\lambda|!}{H(\lambda)},\ \ \text{where}
  \ H(\lambda)\ :=\prod_{(i,j)\in D_\lambda}h(i,j).
\end{equation}
See for example \cite[Corollary 7.21.6]{Stanley_EC2}.

To each partition $\lambda$ of length $r$, we associate a
\textbf{degree vector} $\strict_\lambda$ defined by
\begin{equation}
\label{eq:defdegreevector}
  \strict_{\lambda}:=(n_1,n_2,\dotsc,n_r)
    =(\lambda_r,\ \lambda_{r-1}+1,\ \dotsc,\ \lambda_1+r-1).
\end{equation}
Note its prior appearance in~\eqref{eq:WAP2}. The hook product $H(\lambda)$
has an alternative description in terms of this degree vector; namely,
\begin{equation}
\label{eq:AltHooks}
H(\lambda)=\frac{n_1!\,n_2!\,\cdots\,n_r!}{\Delta(\strict_\lambda)}.
\end{equation}
See for example \cite[Lemma 7.21.1]{Stanley_EC2}. 

\subsection{Symmetric functions}
\label{subsec:symmetricfunctions}
Fix an infinite sequence of variables $X=(x_1,x_2,\dots)$.
The \textbf{ring of symmetric functions} $\Lambda$ consists of all
bounded-degree, integer-coefficient formal series in $X$ that are
invariant under permutations of $X$. Some important examples of
elements in this ring are
\begin{itemize}
\item the \textbf{complete homogeneous symmetric functions}
$(h_m)_{m=1}^\infty$, defined by
\[
h_m = \sum_{i_1\leq i_2 \leq \dots \leq i_m} x_{i_1} x_{i_2} \cdots x_{i_m},
\]
\item the \textbf{elementary symmetric functions} $(e_m)_{m=1}^\infty$,
defined by
\[
e_m = \sum_{i_1<i_2<\dots<i_m} x_{i_1} x_{i_2} \cdots x_{i_m},\ \text{and}
\]
\item the \textbf{power sum symmetric functions} $(p_m)_{m=1}^\infty$,
defined by
\[
p_m = \sum_{i=1}^\infty x_i^m.
\]
\end{itemize}
By convention, $h_0=e_0=1$, whereas $p_0$ is normally left undefined.

It is well-known that $\Lambda$ is freely generated (as a commutative
ring with unit element) by $(h_m)_{m=1}^\infty$
as well as by $(e_m)_{m=1}^\infty$. In other words, every member of
$\Lambda$ is uniquely expressible as a polynomial in $(h_m)_{m=1}^\infty$
as well as in $(e_m)_{m=1}^\infty$, and
\[
\Lambda=\mathbb{Z}[h_1,h_2,\dotsc]=\mathbb{Z}[e_1,e_2,\dotsc].
\]
For the power sums $(p_m)_{m=1}^\infty$, this is not quite true unless we
replace $\Lambda$ with a larger ring, the $\mathbb{Q}$-algebra
$\Lambda_{\mathbb{Q}}$ that allows rational (as opposed to integer)
coefficients; thus,
\[
\Lambda_{\mathbb{Q}}=\mathbb{Q}[p_1,p_2,\dotsc].
\]
For further details, see \cite[I.2]{MacDonald}.

For partitions $\lambda$ of length $r$, one defines
\[
h_\lambda=h_{\lambda_1}h_{\lambda_2}\cdots h_{\lambda_r},\quad
e_\lambda=e_{\lambda_1}e_{\lambda_2}\cdots e_{\lambda_r},\quad
p_\lambda=p_{\lambda_1}p_{\lambda_2}\cdots p_{\lambda_r}
\]
and $h_{\emptyset}=e_{\emptyset}=p_{\emptyset}=1$, 
so that as $\lambda$ varies over $\mathbb{Y}$, $h_\lambda$, $e_\lambda$,
and $p_\lambda$ vary over all of the monomials one can form with the
terms from each of their respective sequences. In this way one sees that
$\{h_\lambda:\lambda\in\mathbb{Y}\}$,
$\{e_\lambda:\lambda\in\mathbb{Y}\}$, and
$\{p_\lambda:\lambda\in\mathbb{Y}\}$
each form bases for $\Lambda_{\mathbb{Q}}$ as a vector space.

There are algebraic relations among these symmetric functions that are
easily expressible as generating function identities. For example,
if we define
\[
H(t):=\sum_{m=0}^\infty h_mt^m,\qquad E(t):=\sum_{m=0}^\infty e_mt^m,
\]
one sees from the definitions of $h_m$ and $e_m$ that
$H(t)=\prod\limits_{i=1}^\infty(1-x_it)^{-1}$
and $E(t)=\prod\limits_{i=1}^\infty(1+x_it)$.
It follows that
\[
\log H(t) = -\log E(-t) = \sum_{i=1}^\infty-\log(1-x_it)
  = \sum_{i=1}^\infty\sum_{m=1}^\infty x_i^m\frac{t^m}{m}
  = \sum_{m=1}^\infty p_m\frac{t^m}{m},
\]
and therefore
\begin{equation}\label{eq:hp-relation}
H(t) = \frac{1}{E(-t)}= \exp\Biggl(\sum_{m=1}^\infty p_m\frac{t^m}{m}\Biggr).
\end{equation}
This shows that the ring automorphism $\omega:\Lambda \to \Lambda$ defined
by setting $\omega(h_m)=e_m$ for $m\ge1$ has the property that
$\omega(p_m)=(-1)^{m-1}p_m$ and $\omega(e_m)=h_m$.
In particular, it is an involution.

A family of symmetric functions of special importance is formed by
the \textbf{Schur functions}
$(s_\lambda)_{\lambda \in \mathbb{Y}}$.
They have many equivalent definitions; the one that is most relevant
for our purposes is the (first) Jacobi-Trudi
formula~\cite[I.3 (3.4)]{MacDonald}
\begin{equation}
\label{eq:firstJacobiTrudi}
s_\lambda = \det[h_{\lambda_i-i+j}]_{1\leq i,j\leq \ell(\lambda)},
\end{equation}
using the convention that $h_{-m}=0$ for integers $m>0$.
This determinant is evidently an integer polynomial in the complete
homogeneous symmetric functions $(h_m)_{m=1}^\infty$, so it is clear from
this definition that each Schur function belongs to $\Lambda$.
It is also not hard to deduce from this definition that the partitions of $m$
may be ordered so that $s_\lambda=h_\lambda+\text{terms $h_\mu$
involving ``later'' }\mu$, so $\{s_\lambda:\lambda\in\mathbb{Y}\}$
is a $\mathbb{Z}$-basis for $\Lambda$ and a $\mathbb{Q}$-basis
for $\Lambda_{\mathbb{Q}}$.

An alternative formula for Schur functions is the
dual Jacobi-Trudi identity~\cite[I.3 (3.5)]{MacDonald},
which amounts to the fact that $\omega(s_\lambda)=s_{\lambda'}$.
In other words, we have
\begin{equation}
\label{eq:secondJacobiTrudi}
s_{\lambda'} = \det[e_{\lambda_i-i+j}]_{1\leq i,j\leq \ell(\lambda)},
\end{equation}
with the similar convention that $e_{-m}=0$ for $m>0$.


\subsection{Appell sequences and Wronskian Appell polynomials}
\label{sec:Appelldefs}

Appell introduced the following family of univariate
polynomial sequences~\cite{Appell}.
\\

\begin{definition}\label{def:Appellsequence}
An \textbf{Appell sequence} is a sequence of polynomials $(A_n)_{n=0}^\infty$ such that
\begin{itemize}
	\item[\textup{(i)}] $A_0=1$, and
	\item[\textup{(ii)}] $A_n'=n A_{n-1}$ for all $n\geq 1$.
\end{itemize}
\end{definition}

An easy consequence of this definition is that each $A_n$ is monic of degree $n$.
Moreover, with $A_1(x) = x + z_1$ the change of variables $x \mapsto x - z_1$ produces a \textbf{central} Appell sequence $(\tilde{A}_n)_{n=0}^\infty$ with $\tilde{A}_1(x)=x$.
Some examples of Appell sequences are the monomials and the probabilists Hermite polynomials.
These and other Appell sequences of interest are discussed in Section \ref{sec:specificAppellsequences}.

For each Appell sequence $(A_n)_{n=0}^\infty$ we set $z_n=A_n(0)$.
One can easily see (by induction and the Appell property) that
\begin{equation}
\label{eq:explicitAppell}
A_n(x)=\sum_{k=0}^{n} \binom{n}{k} z_k \, x^{n-k} \qquad (n\geq 0).
\end{equation}
Furthermore, any Appell sequence has an exponential generating function of the form
\begin{equation}
\label{eq:expgenfunAppell}
A(x,t):=\sum_{k=0}^\infty A_k(x) \frac{t^k}{k!} = \exp(xt) f_A(t),
\end{equation}
where $f_A$ is some formal power series \cite[Section 9]{Carlitz}. 
Substituting $x=0$ in \eqref{eq:expgenfunAppell}, we see that $f_A$ is precisely the exponential generating function of the sequence $(z_n)_{n=0}^\infty$; i.e.,
\begin{equation*}
\label{eq:explicitfA}
f_A(t)=\sum_{k=0}^\infty z_k \frac{t^k}{k!}.
\end{equation*} 
Note that $f_A(0)=z_0=A_0(0)=1$, so one can view the values $z_n$ as the
moments and $f_A(t)$ as the moment generating function of some probability
measure. Building on this analogy, the logarithm of $f_A(t)$ centered
at $t=0$ is
\begin{equation}
\label{eq:logofgenerating}
\log f_A(t)= \sum_{k=1}^\infty c_k \frac{t^k}{k!},
\end{equation}
where the values $c_k$ are the cumulants of this probability measure.
Here, $z_k$ and $c_k$ depend on the specific Appell sequence $(A_n)_{n=0}^\infty$ but we omit this relationship when the Appell sequence is clear from the context. An explicit relation between the values $z_k$ and $c_k$ is given by
\begin{equation*}
	c_n 
	= z_n - \sum_{i=1}^{n-1} \binom{n-1}{i} c_{n-i} \, z_i \qquad (n\geq 1).
\end{equation*}
For more examples and properties of Appell polynomials, we refer to \cite{Aceto} for a matrix approach or \cite{Ta} for a probabilistic approach.

As discussed in the introduction, we define the \textbf{Wronskian Appell
polynomial} associated to a partition $\lambda$ of length $r$ and a given
Appell sequence $(A_n)_{n=0}^\infty$ to be
\[
  A_{\lambda}=\frac{\Wr[A_{n_1},A_{n_2},\dots,A_{n_r}]}{\Delta(\strict_{\lambda})},
\]
where $\strict_\lambda=(n_1,n_2,\dots,n_r)=(\lambda_r,\ \lambda_{r-1}+1,\ \dotsc,\ \lambda_1+r-1)$ as in \eqref{eq:defdegreevector}.

Since each polynomial $A_n$ is monic of degree $n$,
one can show that $A_{\lambda}$ is monic of degree~$|\lambda|$
(see~\cite[Lemma 2.1]{Bonneux_Stevens}).
It is easy to see that $A_{(n)}=A_n$ for all $n\geq 1$ and $A_\emptyset=A_0=1$,
so Wronskian Appell polynomials generalize the Appell sequence. 
One can check that $A_\lambda$ remains unchanged if a~0 is
inserted into the partition $\lambda$. 

%


\section{Wronskian Appell polynomials and Schur functions}
\label{sec:images}

Fix an Appell sequence $A=(A_n)_{n=0}^\infty$.
The main results of this paper rely on a ring
homomorphism $\varphi_A$ from $\Lambda$ to $\mathbb{R}[x]$ defined by
\begin{equation}\label{eq:defconnecting}
  \varphi_A(h_m)=\frac{A_m}{m!} \qquad (m\geq 1).
\end{equation}
This completely determines $\varphi_A$, since $(h_m)_{m=1}^\infty$
freely generates $\Lambda$.

\begin{theorem}
\label{thm:imageschur}
If $A=(A_n)_{n=0}^\infty$ is an Appell sequence and $\lambda \in \mathbb{Y}$,
then
\begin{equation}
\label{eq:imageschur}
\varphi_A(s_\lambda)
   = \frac{A_\lambda}{H(\lambda)}
   = \frac{F_\lambda A_\lambda}{\lvert \lambda \rvert!}.
\end{equation}
\end{theorem}

\begin{proof}
The second equality follows directly from \eqref{eq:Flambdahooklengths}.
For the first equality, let $\ell(\lambda)=r$
and $\strict_\lambda=(n_1,\dots,n_r)$. By the Appell property and \eqref{eq:defconnecting}, we have
\[
A_n^{(j)}=n(n-1)\cdots(n-j+1)A_{n-j}=n!\,\varphi_A(h_{n-j}).
\]
Recalling the convention that $h_m=0$ for $m<0$,
this is valid even for $j>n$. Therefore,
\begin{align*}
A_\lambda
&= \frac{\Wr[A_{n_1},A_{n_2},\dots,A_{n_r}]}{\Delta(\strict_\lambda)}
  =\frac{1}{\Delta(\strict_\lambda)}\cdot
  \det\bigl[n_i!\,\varphi_A(h_{n_i-j+1})\bigr]_{1\le i,j\le r}\\
&= \frac{n_1!\cdots n_r!}{\Delta(\strict_\lambda)}\cdot
  \det\left[\varphi_A(h_{\lambda_{r+1-i}+i-j})\right]_{1\le i,j\le r},
\end{align*}
since $n_i=\lambda_{r+1-i}+i-1$ by definition. Reversing rows
and columns in this determinant yields
\[
A_\lambda = \frac{n_1!\cdots n_r!}{\Delta(\strict_\lambda)}\cdot
  \det\left[\varphi_A(h_{\lambda_i-i+j})\right]_{1\le i,j\le r}
= \frac{n_1!\cdots n_r!}{\Delta(\strict_\lambda)}\varphi_A(s_\lambda)
\]
by~\eqref{eq:firstJacobiTrudi}.
Use~\eqref{eq:AltHooks} to complete the proof.
\end{proof}

The symmetric functions $\tilde{s}_\lambda = H(\lambda) s_\lambda$
are referred to as \textbf{augmented Schur functions}
in~\cite[I.7, Ex.~17(a)]{MacDonald}. Note that~\eqref{eq:imageschur}
directly implies $\varphi_A(\tilde{s}_\lambda)=A_\lambda$.

Since $s_{(1^n)}=e_n$, which is the $1\times 1$ case of the second Jacobi-Trudi identity \eqref{eq:secondJacobiTrudi}, we have the following corollary.
 
\begin{corollary}
\label{cor:imageelementary}
If $A=(A_n)_{n=0}^\infty$ is an Appell sequence, then for $n\geq0$,
\[
\varphi_A(e_n)=\frac{A_{(1^n)}}{n!}.
\]
\end{corollary}

In Section \ref{sec:dualappell}, we study the polynomials $(A_{(1^n)})_{n=0}^{\infty}$ and show they also form an Appell sequence.
We now consider the image of the power sum symmetric functions $p_n$.

\begin{proposition}
\label{prop:imageofpowersums}
If $(A_n(x))_{n=0}^\infty$ is an Appell sequence and $c_n$ is given
as in~\eqref{eq:logofgenerating}, then
\[
\varphi_A(p_1)=x+c_1 \quad \text{and} \quad
   \varphi_A(p_n)= \frac{c_n}{(n-1)!} \quad \text{for }\, n \geq 2.
\]
\end{proposition} 

\begin{proof}
Applying $\phi_A$ to~\eqref{eq:hp-relation} yields
\[
\exp\Biggl(\sum_{n=1}^\infty\varphi_A(p_n)\frac{t^n}{n}\Biggr)
  = \sum_{n=0}^\infty \varphi_A(h_n)t^n
  = \sum_{n=0}^\infty A_n(x)\frac{t^n}{n!}=e^{xt}f_A(t),
\]
the last equality being~\eqref{eq:expgenfunAppell}. Hence
\[
\sum_{n=1}^\infty \varphi_A(p_n) \frac{t^n}{n}
   = xt + \log f_A(t) = (x+c_1)t \ +\  \sum_{n=2}^\infty c_n \frac{t^n}{n!},
\]
from which the result follows.
\end{proof}

It is noteworthy that the value of $\varphi_A(p_n)$ is a constant for $n\geq 2$ by the above proposition. We now have the following table of images of the homomorphism $\varphi_A$.

\begin{center}
\renewcommand{\arraystretch}{1.5}
\begin{tabular}{c c c c}
$\Lambda$ & $\xrightarrow{\varphi_{A}}$ & $\mathbb{R}[x]$ & \\
\hline
$s_\lambda$ & & $F_\lambda A_\lambda/\lvert \lambda \rvert!$ & \\
$\tilde{s}_\lambda$ & & $A_\lambda$ & \\
$h_n$ & & $A_n/n!$ & \\
$e_n$ & & $A_{(1^n)}/n!$ & \\
$p_1$ & & $x+c_1$ & \\
$p_n$ & & $c_n/(n-1)!$ & if $n\geq 2$\\
\hline
\end{tabular}
\end{center}


\section{Some consequences}
\label{sec:corollaries}
\subsection{The derivative and Appell nets}
\label{sec:derivatives}
In~\cite{Bonneux_Stevens}, the first and last authors observed a generalization of the Appell property for Wronskian Hermite polynomials.
The following extends their result to all Wronskian Appell polynomials.

\begin{theorem}\label{thm:WAPisAN}
If $(A_n)_{n=0}^\infty$ is an Appell sequence,
then the polynomials $(A_\lambda)_{\lambda \in \mathbb{Y}}$ satisfy
\begin{equation}
\label{eq:AppellNet}
F_\lambda^{\vphantom'}A_\lambda' 
	= \lvert \lambda \rvert \sum_{\mu \covby \lambda} F_\mu A_\mu.
\end{equation}
\end{theorem}

\begin{proof}
Since $(p_\lambda)_{\lambda \in \mathbb{Y}}$ is a basis of $\Lambda_{\mathbb{Q}}$, we may regard each $f\in\Lambda_{\mathbb{Q}}$ as a polynomial in finitely many of the variables $(p_n)_{n=1}^\infty$. In this way, $f$ has well-defined partial derivatives with respect to these variables.
In particular, we claim that
\begin{equation}
\label{eq:partial-phi}
\frac{\partial}{\partial x}\varphi_A(f)
  = \varphi_A\Bigl(\frac{\partial f}{\partial p_1}\Bigr).
\end{equation}
Since both sides are linear, it suffices to check this for
monomials $f=p_1^{\alpha_1}\cdots p_m^{\alpha_m}$. In that case,
Proposition~\ref{prop:imageofpowersums} implies
\[
\frac{\partial}{\partial x}\varphi_A(f)
  = \alpha_1(x+c_1)^{\alpha_1-1}c_2^{\alpha_2}\cdots c_m^{\alpha_m}
  = \varphi_A\Bigl(\frac{\partial f}{\partial p_1}\Bigr),
\]
proving the claim.

On the other hand, formally differentiating~\eqref{eq:hp-relation}
with respect to $p_1$ yields
\[
\frac{\partial H(t)}{\partial p_1}=tH(t),
\]
and thus $\partial h_n/\partial p_1 = h_{n-1}$ for all $n$
(with $h_{-m}=0$ for $m>0$ as usual).
It follows that by differentiation
of~\eqref{eq:firstJacobiTrudi}, one obtains
\[
\frac{\partial}{\partial p_1} s_\lambda
 \ =\ \sum_{k=1}^{\ell(\lambda)}\det\left[h_{\lambda_i-\delta_{k,i}-i+j}\right]_{1\leq i,j\leq \ell(\lambda)},
\]
where $\delta_{k,i}$ denotes a Kronecker delta.
If $\lambda_k>\lambda_{k+1}$, the $k$th determinant in this sum is $s_\mu$,
where $\mu$ is obtained from $\lambda$ by decreasing $\lambda_k$ by~1.
Otherwise, if $\lambda_k=\lambda_{k+1}$, then the $k$th determinant has
two equal rows and therefore vanishes.
Thus the nonzero terms in the sum are indexed precisely by the partitions
covered by $\lambda$ in Young's lattice, and we conclude that
\[
\frac{\partial}{\partial p_1} s_\lambda = \sum_{\mu \covby \lambda} s_\mu.
\]
The result now follows from~\eqref{eq:partial-phi}
and Theorem~\ref{thm:imageschur}.
\end{proof}

Motivated by the above result, define an \textbf{Appell net} to be a collection of univariate polynomials $(A_\lambda)_{\lambda \in \mathbb{Y}}$
with $A_\emptyset = 1$ satisfying~\eqref{eq:AppellNet}.
Setting $z_\lambda = A_\lambda(0)$ and $n=|\lambda|$,
a $k$-fold iteration of~\eqref{eq:AppellNet} yields
\[
F_\lambda A^{(k)}_\lambda = n(n-1)\cdots(n-k+1)
  \sum_{\mu\vdash n-k}F_{\lambda/\mu} F_\mu A_\mu,
\]
and therefore
\begin{equation*}
\label{eq:Appellnetexpression}
F_{\lambda} A_{\lambda}(x)
  = \sum_\mu\binom{|\lambda|}{|\mu|}F_{\lambda/\mu}
     F_{\mu} z_{\mu} x^{|\lambda|-|\mu|}.	
\end{equation*}
Conversely, it is not hard to check that every choice of constants $(z_\lambda)_{\lambda \in \mathbb{Y}}$ yields an Appell net via the above formula. Moreover, a similar construction leads to Appell nets on any differential poset as introduced by Stanley \cite{Stanley_DiffPos}. However, we will not pursue this further here.

\subsection{Plancherel measure statistics}
\label{sec:averages}
It is well known that the order of a finite group is the sum of the
squares of the dimensions of its irreducible representations.
In the case of the symmetric group of degree $n$, this identity
takes the form
\begin{equation}\label{eq:Plancherel}
	n! = \sum_{\lambda \vdash n} F^2_\lambda.
\end{equation}
The corresponding \textbf{Plancherel measure} may thus be viewed as
a probability measure on partitions of $n$ with
\[
\mathbb{P}(X = \lambda) = \frac{1}{n!}F^2_\lambda.
\]
In particular, for each Appell sequence $(A_m)_{m=0}^\infty$ and each $n$
we may interpret the associated Wronskian Appell polynomials
$\{A_\lambda:\lambda\vdash n\}$ as a random variable with respect to
this measure.

In the following, we derive the expected value and variance of these
random variables. The first of these
generalizes~\cite[Theorem 3.4]{Bonneux_Stevens} from the Hermite
case to any Appell sequence.

\begin{theorem}
\label{thm:average}
If $(A_m)_{m=0}^\infty$ is an Appell sequence, then
$\mathbb{E}(A_\lambda:\lambda\vdash n) = A_1^n$.
\end{theorem}

\begin{proof}
If we apply $\varphi_A$ to the identity~\cite[Corollary 7.12.5]{Stanley_EC2}
\[
\sum_{\lambda \vdash n} F_\lambda s_\lambda = h_1^n,
\]
the result follows.
\end{proof}

For the second moment, recall that an Appell sequence
$(A_m)_{m=0}^\infty$ is central if $A_1(0)=0$.

\begin{theorem}	\label{thm:secondmoment}
If $A=(A_m)_{m=0}^\infty$ is the Appell sequence determined
by $\displaystyle\log f_A(t) = \sum\limits_{k=1}^\infty c_k \frac{t^k}{k!}$, then
\begin{equation*}\label{eq:secondmoment}
\mathbb{E}\left(A^2_{\lambda}(x):\lambda\vdash n\right)
  = \sum_{\lambda \vdash n}\frac{1}{n!}F_\lambda^2 A^2_\lambda(x) 
	= B_n((x+c_1)^2),
\end{equation*}
where $B=(B_n)_{n=0}^{\infty}$ is the central Appell sequence determined by $\displaystyle\log f_B(t) = \sum\limits_{k=2}^\infty \frac{c_k^2}{(k-1)!} \frac{t^k}{k!}$.
\end{theorem}

\begin{proof}
By Theorem~\ref{thm:imageschur}, we have
\[
\sum_{n=0}^\infty\mathbb{E}\left(A^2_{\lambda}(x):\lambda\vdash n\right)
  \frac{t^n}{n!}
=\sum_{n=0}^\infty\sum_{\lambda\vdash n}F_\lambda^2A_\lambda(x)^2
  \frac{t^n}{(n!)^2}
=\varphi_A\biggl(\sum_\lambda s^2_\lambda t^{|\lambda|}\biggr).
\]
On the other hand, by specializing the Cauchy identity~\cite[I.4 (4.3)]{MacDonald}, we have
\[
\sum_\lambda s^2_\lambda t^{|\lambda|}
  \ =\ \prod_{i=1}^\infty\prod_{j=1}^\infty\frac{1}{1-x_ix_jt}
  \ =\ \exp\biggl(\sum_{k=1}^\infty p_k^2\frac{t^k}{k}\biggr),
\]
the second equality following by the same reasoning we used to prove~\eqref{eq:hp-relation} (cf.~\cite[I.4~(4.1)]{MacDonald}). 
Applying Proposition~\ref{prop:imageofpowersums}, we obtain
\[
\sum_{n=0}^\infty\mathbb{E}\left(A^2_{\lambda}(x):\lambda\vdash n\right)
  \frac{t^n}{n!}
=\exp\biggl( (x+c_1)^2t + \sum_{k=2}^\infty\frac{c_k^2}{(k-1)!}\frac{t^k}{k!}\biggr)
=e^{(x+c_1)^2t}f_B(t),
\]
where $B$ is the Appell sequence defined above.
Since $e^{xt}f_B(t)$ is the exponential generating function for this
sequence, we obtain the claimed result by extracting the coefficient
of $t^n$ in the above identity.
\end{proof}

As a corollary, we have the following property of the variance.

\begin{corollary}\label{cor:variance}
If $A=(A_m)_{m=0}^\infty$ is an Appell sequence, then for $n\ge2$,
\begin{equation*}
\Var\left(A_\lambda(x):\lambda\vdash n\right)
   = \mathbb{E}\left(A^2_{\lambda}(x):\lambda\vdash n\right)
   - \mathbb{E}\bigl(A_{\lambda}(x):\lambda\vdash n\bigr)^2 = O(x^{2n-4}).
\end{equation*}
\end{corollary}

\begin{proof}
Combining Theorems~\ref{thm:average} and~\ref{thm:secondmoment}, we have
\begin{equation*}
	\Var(A_\lambda(x):\lambda\vdash n)
		= B_n((x+c_1)^2) - (x+c_1)^{2n}.
\end{equation*}
On the other hand, since $B$ is central, we
have $B_n(x)=x^n+O(x^{n-2})$ for $n\geq2$ (recall~\eqref{eq:explicitAppell}),
and the result follows.
\end{proof}

Higher moments can be derived from~\cite[Ex 7.70]{Stanley_EC2}, though a closed formula appears unlikely except in special cases.

\subsection{Dual Appell sequences}
\label{sec:dualappell}
Recall from the discussion at the end of
Section~\ref{subsec:symmetricfunctions}
that there is a ring involution $\omega$ of $\Lambda$ such that
$\omega(h_m)=e_m$, $\omega(p_m)=(-1)^{m-1}p_m$,
and $\omega(s_\lambda)=s_{\lambda'}$. This allows us to deduce
that for any Appell sequence $A$, there is a second Appell sequence $A^*$
hidden within the associated net of Wronskian Appell polynomials
generated by $A$.

\begin{theorem}
\label{thm:dualWronskian}
If $A=(A_n)_{n=0}^\infty$ is an Appell sequence, then
\begin{itemize}
\item[\textup{(a)}] the sequence $A^*=(A^*_n)_{n=0}^\infty$ defined by setting $A^*_n=A_{(1^n)}$ is also an Appell sequence,
\item[\textup{(b)}] the Wronskian Appell polynomials for $A$ and $A^*$ satisfy $A^*_\lambda=A_{\lambda'}$ for all $\lambda\in\mathbb{Y}$, and
\item[\textup{(c)}] the exponential generating functions $f_A(t)$ and $f_{A^*}(t)$ are related by $f_{A^*}(t)=1/f_A(-t)$, or equivalently,
\[
\log f_A(t)=\sum_{k=1}^\infty c_k\frac{t^k}{k!}
\quad\text{implies}\quad
\log f_{A^*}(t)=\sum_{k=1}^\infty (-1)^{k-1}c_k\frac{t^k}{k!}.
\]
\end{itemize}
\end{theorem}

\begin{proof}
For every $n\ge0$, we have $F_{(1^n)}=1$. Furthermore, the only partition
that is covered by $(1^n)$ in Young's lattice is $(1^{n-1})$.
Theorem~\ref{thm:WAPisAN} therefore implies
\[
(A_n^*)'=A_{(1^n)}'=n A_{(1^{n-1})}=n A_{n-1}^*
  \qquad(n\ge1),
\]
proving~(a). Now consider that
Corollary~\ref{cor:imageelementary} implies
\[
(\varphi_A \circ \omega)(h_n)=\varphi_A(e_n)
   = \frac{A_n^*}{n!}=\varphi_{A^*}(h_n).
\]
Since $(h_n)_{n=0}^\infty$ generates $\Lambda$, it follows more generally
that $\varphi_{A^*}(g) = (\varphi_A \circ \omega)(g)$ for all $g\in\Lambda$.
Recalling that $\omega(s_\lambda)=s_{\lambda'}$ for all $\lambda\in\mathbb{Y}$, we obtain
\[
\frac{F_\lambda A_\lambda^*}{\lvert \lambda \rvert!}
 = \varphi_{A^*}(s_\lambda)
 = (\varphi_A \circ \omega)(s_\lambda)
 = \varphi_A(s_{\lambda'})
 = \frac{F_{\lambda'} A_{\lambda'}}{\lvert \lambda' \rvert!}.
\]
Since $F_{\lambda'}=F_\lambda$ and $\lvert \lambda'\rvert=\lvert \lambda \rvert$, this proves~(b).
Part~(c) follows by applying $\varphi_A$ to~\eqref{eq:hp-relation}.
\end{proof}

We call $A^*$ the Appell sequence \textbf{dual} to $A$.

If we specialize Theorem~\ref{thm:dualWronskian} to the setting of Hermite polynomials, we recover \cite[Theorem~1.2]{Curbera_Duran} and \cite[Corollary 2]{GomezUllate_Grandati_Milson-16}.
Moreover, similar results are derived in other settings as well, for example see \cite[Theorem 6.1 and Theorem 8.1]{Curbera_Duran} or \cite[Lemma 2.7]{Bonneux} and \cite[Lemma 5]{Bonneux_Kuijlaars}. These cited results relate a Wronskian corresponding to two partitions to the Wronskian corresponding to the conjugated partitions. Similar identities with a combinatorial interpretation in terms of Maya diagrams can be found in \cite{GomezUllate_Grandati_Milson-16,GomezUllate_Grandati_Milson-L+J}.

%
\begin{corollary}
\label{cor:doubledual}
If $(A_n)_{n=0}^\infty$ is an Appell sequence, then
$A_n^{**}=A_n$ for all $n\geq0$.
\end{corollary}

\begin{corollary}\label{cor:dual}
If $(A_n)_{n=0}^{\infty}$ is an Appell sequence,
the following statements are equivalent:
\begin{itemize}
	\item [\textup{(a)}] For all integers $n\geq0$, $A_n=A^*_n$, i.e. the Appell sequence is self-dual.
	\item [\textup{(b)}] For all $\lambda\in\mathbb{Y}$, $A_{\lambda}=A_{\lambda'}$.
	\item[\textup{(c)}] For all even integers $n>0$, $c_n=0$.
\end{itemize}
\end{corollary}

\subsection{Integer coefficients}
\label{sec:integercoefficients}

Previously, the first and last authors conjectured that Wronskian Hermite polynomials have integer coefficients~\cite[Conjecture 3.7]{Bonneux_Stevens}.
The following result provides a sufficient condition on Appell polynomials
so that the associated Wronskian Appell polynomials have integer coefficients;
it confirms their conjecture as a special case (see the discussion in Section~\ref{sec:specificAppellsequences}).

\begin{theorem}
\label{thm:integercoefficients}
Let $(A_n)_{n=0}^\infty$ be an Appell sequence with $c_k$  as in \eqref{eq:logofgenerating}. If $c_k/(k-1)! \in \mathbb{Z}$ for all $k\ge1$, then $A_\lambda \in \mathbb{Z}[x]$ for all $\lambda\in\mathbb{Y}$.
\end{theorem}

\begin{proof}
Fix $\lambda \vdash n$ and recall from Theorem \ref{thm:imageschur} that $\varphi_{A}(H(\lambda) s_\lambda) = A_\lambda$.
It is known that if the Schur function $s_\lambda$ is rescaled by the
factor $H(\lambda)$, the result is a polynomial in power sums with
integer coefficients (see \cite[I.7 Ex.~17(a)]{MacDonald}). In other words,
there exist integers $d_{\lambda\mu}$ such that
\[
H(\lambda)s_\lambda = \sum_{\mu \vdash n} d_{\lambda\mu} p_\mu.
\]
On the other hand, Proposition~\ref{prop:imageofpowersums} and the
stated hypothesis imply that $\varphi_A(p_1)=x+c_1\in\mathbb{Z}[x]$
and $\varphi_A(p_k)=c_k/(k-1)!\in\mathbb{Z}$ for $k>1$.
\end{proof}

\section{Rim hooks and recurrence relations}
\label{sec:rimandrecurrence}

A \textbf{rim hook} of size $k$ is a skew diagram $\lambda/\mu$ that is connected and
does not contain any $2 \times 2$ squares such that $|\lambda|-|\mu|=k$.
Its \textbf{height}, denoted $\htt(\lambda/\mu)$,
is one less than the number of rows it occupies. We set
\begin{align*}
\rhup{k}{\mu}&:=\{\lambda\in\mathbb{Y}:\lambda/\mu\text{ is a rim hook of size $k$}\},\\
\rhdown{k}{\lambda}&:=\{\mu\in\mathbb{Y}:\lambda/\mu\text{ is a rim hook of size $k$}\}.
\end{align*}
Note that a rim hook $\lambda/\mu$ of size~1 is simply a covering
pair in Young's lattice.

A rim hook with a fixed outer shape $\lambda$ has at most one cell on
each northwest-southeast diagonal and thus is determined by the highest
row and leftmost column it occupies.
Conversely, for each cell $(i,j)\in D_\lambda$, there is a rim
hook $\lambda/\mu$ with highest row $i$ and leftmost column $j$,
and its size is the hook length $h(i,j)$.
This bijection between the cells of $D_\lambda$ and $\lambda$-bounded
rim hooks is illustrated below for $\lambda=(5,3,2)$ where cells are
marked with a bullet and rim hooks are shaded gray.

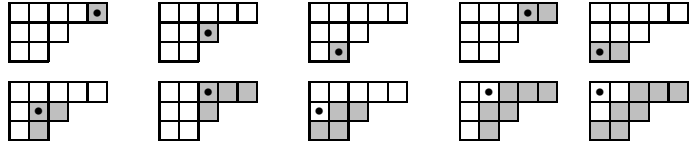
\begin{center}	
	\begin{footnotesize}
	\begin{tabular}{l l l l l}
\ydiagram[*(lightgray) \bullet]{4+1,3+0,2+0}
		*[*(white)]{5,3,2}
		*[*(white) ]{4+1,3+0,2+0}
		&
		\ydiagram[*(lightgray) \bullet]{5+0,2+1,2+0}
		*[*(white)]{5,3,2}
		&
		\ydiagram[*(lightgray) \bullet]{5+0,3+0,1+1}
		*[*(white)]{5,3,2}
		\vspace{0.25cm}
		&
		\ydiagram[*(lightgray) \bullet]{3+1,3+0,2+0}
		*[*(lightgray)]{3+2,3+0,2+0}
		*[*(white)]{5,3,2}
		&
		\ydiagram[*(lightgray) \bullet]{5+0,3+0,0+1}
		*[*(lightgray)]{5+0,3+0,0+2}
		*[*(white)]{5,3,2}
		\\
		\ydiagram[*(lightgray) \bullet]{5+0,1+1,2+0}
		*[*(lightgray)]{5+0,1+2,1+1}
		*[*(white)]{5,3,2}
		&
		\ydiagram[*(lightgray) \bullet]{2+1,3+0,2+0}
		*[*(lightgray)]{2+3,2+1,2+0}
		*[*(white)]{5,3,2}
		&
		\ydiagram[*(white) \bullet]{5+0,0+1,2+0}
		*[*(lightgray)]{5+0,1+2,0+2}
		*[*(white)]{5,3,2}
		&
        \ydiagram[*(white) \bullet]{1+1,3+0,2+0}
		*[*(lightgray)]{2+3,1+2,1+1}
		*[*(white)]{5,3,2}
		&\ydiagram[*(white) \bullet]{0+1,3+0,2+0}
		*[*(lightgray)]{2+3,1+2,0+2}
		*[*(white)]{5,3,2}
	\end{tabular}
	\end{footnotesize}
	\end{center}
The rim hooks in the first row above all have height~0,
while the first three in the second row have height 1 and the last two
have height~2.

The Murnaghan-Nakayama Rule~\cite[I.7 Ex. 5]{MacDonald}
uses rim hooks to provide a combinatorial formula for multiplying a Schur function by a power sum; namely,
\begin{equation}\label{eq:MNRule}
p_k s_\lambda \ = \sum_{\gamma\in\rhup{k}\lambda}
  (-1)^{\htt(\gamma/\lambda)} s_\gamma.
\end{equation}
The following identity may be viewed as providing a one-sided inverse to~\eqref{eq:MNRule}. We have not seen it elsewhere in the literature on symmetric functions.

\begin{theorem}
\label{thm:symmetricgeneratingrecurrence}
For $\lambda \vdash n$, we have
\begin{equation}
\label{eq:SchurNewton}
n s_\lambda = \sum_{k=1\vphantom{\rhdown{k}\mu}}^n
  \ \sum_{\mu\in\rhdown{k}\lambda} (-1)^{\htt(\lambda/\mu)} p_ks_\mu.
\end{equation}
\end{theorem}

Note that the special case $\lambda=(1^n)$ of~\eqref{eq:SchurNewton}
yields Newton's identities; namely,
\[
n e_n = \sum_{k=1}^n (-1)^{k-1} p_ke_{n-k}.
\]

\begin{proof}
Let $\iprod{\cdot\,{,}\,\cdot}$ denote the standard inner product
on $\Lambda_{\mathbb{Q}}$ relative to which the Schur functions
$(s_\nu)_{\nu\in\mathbb{Y}}$ are orthonormal.
As in the proof of Theorem~\ref{thm:WAPisAN}, we may regard each
$f\in\Lambda_{\mathbb{Q}}$ as a polynomial in finitely many of
the variables $(p_k)_{k=1}^\infty$ and apply differential operators
with respect to such variables. In these terms, it is known
(see~\cite[I.5 Ex.~3(c)]{MacDonald}) that the
operator $k\partial/\partial p_k$ is adjoint to multiplication by $p_k$; i.e.,
\[
\iprod[\Big]{k\frac{\partial f}{\partial p_k}\,,\, g}\ = \iprod{f\,,\, p_kg}
  \qquad (f,g\in\Lambda_{\mathbb{Q}}).
\]
Applying~\eqref{eq:MNRule}, we obtain
\begin{equation}\label{eq:diffSchur}
k\frac{\partial s_\lambda}{\partial p_k}
  \ = \sum_{\mu\vdash n-k}\iprod[\Big]{k\frac{\partial s_\lambda}{\partial p_k}\,,\, s_\mu}s_\mu
  \ = \sum_{\mu\vdash n-k}\iprod{s_\lambda\,,\,p_ks_\mu}s_\mu
  \ = \sum_{\mu\in\rhdown{k}\lambda}(-1)^{\htt(\lambda/\mu)}s_\mu.
\end{equation}
On the other hand, any $p$-monomial $f=p_1^{m_1}p_2^{m_2}\cdots$ is
homogeneous of degree $n=\sum km_k$, from which it follows directly that
\[
nf = \sum_{k=1}^nkp_k\frac{\partial f}{\partial p_k},
\]
and hence also for any $f\in\Lambda_{\mathbb{Q}}$ that is homogeneous of degree~$n$. Specializing to the case $f=s_\lambda$ and applying~\eqref{eq:diffSchur} yields the claimed identity.
\end{proof}

\subsection{Top-down relations}
\label{sec:topdowns}

The following result generalizes \cite[Theorem 3.2]{Bonneux_Stevens}.

\begin{theorem}
\label{thm:topdown}
If $(A_n)_{n=0}^\infty$ is an Appell sequence with $c_n$ as in~\eqref{eq:logofgenerating}, then for all $\lambda \in \mathbb{Y}$,
\begin{align}
(\lvert \lambda \rvert +1) (x+c_1) F_\lambda A_\lambda &=\ \ \sum_{\gamma \cov \lambda} F_\gamma A_\gamma,\label{eq:topdown1}\\
k c_k \binom{\lvert \lambda \rvert +k}{k} F_\lambda A_\lambda &= \sum_{\gamma\in\rhup{k}\lambda} (-1)^{\htt(\gamma/\lambda)} F_\gamma A_\gamma \qquad (k \geq 2).\label{eq:topdown2}
\end{align}
\end{theorem}

\begin{proof}
Apply $\varphi_A$ to~\eqref{eq:MNRule} and use Theorem~\ref{thm:imageschur}
and Proposition~\ref{prop:imageofpowersums} to simplify the result.
\end{proof}

\subsection{The generating recurrence relation}
\label{sec:generatingrecurrence}

In previous work, the first and last authors proved identities for Wronskian Hermite polynomials using a relation they referred to as the generating recurrence relation~\cite[Theorem 3.1]{Bonneux_Stevens}. The following result extends it to all Appell sequences.

\begin{theorem}
\label{thm:generatingrecurrence}
If $(A_n)_{n=0}^\infty$ is an Appell sequence with $c_k$ as
in~\eqref{eq:logofgenerating} and $\lambda \vdash n \geq 1$, then
\begin{equation}
\label{eq:generatingrecursion}
F_\lambda A_\lambda = x \sum_{\mu \covby \lambda} F_\mu A_\mu \ +\ \sum_{k=1}^{n} c_k\binom{n -1}{k-1} \sum_{\nu\in\rhdown{k}\lambda} (-1)^{\htt(\lambda/\nu)} F_\nu A_\nu.
\end{equation}
\end{theorem}

\begin{proof}
Apply $\varphi_A$ to~\eqref{eq:SchurNewton} and use
Theorem~\ref{thm:imageschur} and Proposition~\ref{prop:imageofpowersums} to obtain
\[
\frac{F_\lambda}{(n-1)!} A_\lambda  = (x+c_1) \sum_{\mu \covby \lambda}  \frac{F_\mu}{(n-1)!} A_\mu + \sum_{k=2}^n \frac{c_k}{(k-1)!} \sum_{\nu\in\rhdown{k}\lambda} (-1)^{\htt(\lambda/\nu)}\frac{F_\nu}{(n-k)!} A_\nu.
\]
The claimed result now follows after rearranging the terms.
\end{proof}

\section{Results for specific Appell sequences}
\label{sec:specificAppellsequences}

As discussed in the introduction, one of the main motivations for studying Wronskian Appell polynomials is their appearance in the rational solutions of Painlev\'e equations and in the study of exceptional orthogonal polynomials. All of these appearances involve the use of specific choices of Appell sequences and specific partitions. In this section we examine the consequences of our results for several Appell sequences of interest, with an emphasis on the sequences relevant for these applications.

%

\subsection{Wronskians of monomials}
The simplest example of an Appell sequence is the monomial sequence
$M=(x^n)_{n=0}^\infty$.
The exponential generating function of this sequence is $f_M(t)=1$, so $\log f_M(t)=0$.
Therefore $z_0=1$ and $z_k=c_k=0$ for all $k\geq 1$.
By direct computation, it is not hard to check that
\begin{equation*}
\label{eq:Wronskianmonomial}
M_\lambda(x)=x^{\lvert \lambda \rvert},
\end{equation*}
for any partition $\lambda$.
Therefore Theorem~\ref{thm:average} reduces to~\eqref{eq:Plancherel}.
The recurrence relations \eqref{eq:topdown1}--\eqref{eq:topdown2} and \eqref{eq:generatingrecursion} also simplify to the well-known identities
\begin{equation*}
	(\lvert \lambda \rvert+1) F_\lambda = \sum_{\gamma \cov \lambda} F_\gamma,
	\qquad \qquad F_\lambda = \sum_{\mu \covby \lambda} F_\mu,
\end{equation*}
and the not so well-known
\begin{equation*}
  \sum_{\gamma\in\rhup{k}\lambda} (-1)^{\htt(\gamma/\lambda)} F_\gamma = 0 \qquad (k\geq 2).
\end{equation*}
By Corollary~\ref{cor:dual}, the sequence $(x^n)_{n=0}^\infty$ is self-dual.
The conditions of Theorem~\ref{thm:integercoefficients} are also satisfied,
although the integrality of the Wronskian monomials is trivial.

\subsection{Yablonskii-Vorobiev polynomials (P-II) and Wronskians of Hermite polynomials (P-IV)}
The rational solutions of the Painlev\'e II and the Painlev\'e IV equation can be described in terms of a Wronskian of the polynomials given by the generating series
\begin{equation*}
	\text{P-II: }\ \sum_{k=0}^\infty p_k(x) \frac{t^k}{k!} = \exp\left(xt - \tfrac{4}{3}t^3\right),
	\qquad
	\text{P-IV: }\ \sum_{k=0}^\infty \He_k(x) \frac{t^k}{k!}=\exp\left(xt - \tfrac{1}{2}t^2\right).
\end{equation*}
These Wronskians (for specific partitions) are called Yablonskii-Vorobiev polynomials (P-II) and generalized Hermite polynomials and generalized Okamoto polynomials (P-IV), see \cite[Section 6.1.1 and 6.1.3]{VanAssche}.
Both sequences have generating series of the form
\begin{equation}\label{eq:YaHeclassofpolynomials}
	\sum_{k=0}^\infty A_k(x) \frac{t^k}{k!} = \exp\left(xt + \alpha t^r\right),
\end{equation}
where $\alpha\in\mathbb{R}$ and $r$ is some positive integer. 
With $\alpha=0$ we recover the monomials, and with $r=1$ we obtain translated monomials $A_n(x)=(x+\alpha)^n$.
If $A=(A_n)_{n=0}^\infty$ is the Appell sequence satisfying \eqref{eq:YaHeclassofpolynomials}, then
\begin{equation*}
	f_A(t)=\exp(\alpha t^r), \qquad \log f_A(t)=\alpha t^r.
\end{equation*}
Therefore $z_k=0$ unless $k$ is a multiple of $r$, in which case $z_k=k!\cdot\alpha^{k/r}/(k/r)!$, whereas $c_k=0$ for $k\neq r$ and $c_r=r!\cdot\alpha$.
These polynomials obey the recurrence 
\begin{equation*}\label{eq:YaHereccurence}
	A_n(x) = x A_{n-1}(x) + r\alpha \frac{(n-1)!}{(n-r)!} A_{n-r}(x)\qquad(n \geq r),
\end{equation*}
along with the initial conditions $A_n(x)=x^n$ for $0\le n<r$.

The corresponding dual Appell sequence has generating series $f_{A^*}(t)=\exp((-1)^{r-1}\alpha t^r)$ (see Theorem~\ref{thm:dualWronskian}(c)), so this class of polynomials is closed under taking duals.
In particular if $\alpha\neq 0$, then $A$ is self-dual if and only if $r$ is odd.
A simple calculation gives
\begin{equation*}\label{eq:YaHetransformationfordual}
	A_n^*(x)=\rho^n A_n(\rho^{-1}x),
\end{equation*}
where $\rho=-\exp(\pi i/r)$. In turn this yields the relation
\begin{equation}\label{eq:YaHeWronskiandualtransformation}
	A_{\lambda'}(x)=\rho^{\lvert \lambda \rvert} A_\lambda(\rho^{-1} x).
\end{equation}
When considering the Hermite polynomials, $r=2$ and $\alpha=-1/2$, then \eqref{eq:YaHeWronskiandualtransformation} reduces to what is already known; see for example \cite{Bonneux_Stevens,Duran-Hermite}.

Theorem \ref{thm:integercoefficients} implies $A_\lambda$ has integer coefficients if $r\alpha$ is an integer. Again specializing to the case of Hermite polynomials, this proves Conjecture 3.7 in \cite{Bonneux_Stevens}.

\begin{corollary}
\label{cor:HeZ}
For any partition $\lambda$, we have $\He_\lambda\in \mathbb{Z}[x]$.
\end{corollary}

The polynomials used for constructing the Yablonskii-Vorobiev polynomials have $r=3$ and $\alpha=-4/3$, thus $r\alpha$ is again an integer, and the Wronskian polynomials again have integer coefficients.

The generating recurrence relation (Theorem~\ref{thm:generatingrecurrence}) specializes to
\begin{equation*}\label{eq:YaHegeneratingrecurrence}
	F_\lambda A_\lambda = x \sum_{\mu \covby \lambda} F_\mu A_\mu \ +\ r\alpha \frac{(\lvert \lambda \rvert-1)!}{(\lvert \lambda \rvert -r)!} \sum_{\nu\in\rhdown{r}\lambda} (-1)^{\htt(\lambda/\nu)} F_\nu A_\nu.
\end{equation*}
Setting $\alpha=-1/2$ and $r=2$, we recover the generating recurrence relation for the Hermite polynomials \cite{Bonneux_Stevens}.

For $r,k>1$, the top-down relations (Theorem~\ref{thm:topdown}) specialize to  
\begin{align*}\label{eq:YaHektopdown}
	r\cdot r! \cdot \alpha \binom{\lvert \lambda \rvert +r}{r} F_\lambda A_\lambda &= \sum\limits_{\gamma\in\rhup{k}\lambda} (-1)^{\htt(\gamma/\lambda)} F_\gamma A_\gamma &&\text{if }k=r, \\
	0 &= \sum\limits_{\gamma\in\rhup{k}\lambda} (-1)^{\htt(\gamma/\lambda)} F_\gamma A_\gamma  && \text{if }k\neq r.
\end{align*}

If $r>1$, the average Wronskian polynomial (with respect to the Plancherel measure) equals the monomial $x^{|\lambda|}$.
Theorem~\ref{thm:secondmoment} allows us to compute the second moment.

\begin{corollary}
\label{thm:YaHe2moment}
Fix $r>1$ and $\alpha\in\mathbb{R}$.
If $A=(A_n)_{n=0}^\infty$ is as in \eqref{eq:YaHeclassofpolynomials} with parameters $\alpha$ and $r$, and $B=(B_n)_{n=0}^\infty$ is as in~\eqref{eq:YaHeclassofpolynomials} with parameters $r\alpha^2$ and $r$, then
\begin{equation*}\label{eq:YaHe2moment}
  \mathbb{E}\left(A^2_{\lambda}(x):\lambda\vdash n\right)
    = \sum_{\lambda\vdash n} \frac{F_{\lambda}^2}{n!} A_{\lambda}^2(x)
    = B_n(x^2).
\end{equation*}
\end{corollary}

\begin{proof}
Given that $r>1$, we have that $A$ is a central Appell sequence with $c_r=r!\cdot\alpha$ and $c_k=0$ for $k\neq r$. In particular, $c_r^2/(r-1)! = r \cdot r!\cdot\alpha^2$. Now apply Theorem~\ref{thm:secondmoment}.
\end{proof}

In the special case of Hermite polynomials, the Appell sequence $B$ is
the dual of $A$, so
\begin{equation*}\label{eq:Hermite2moment}
  \sum_{\lambda\vdash n} \frac{F_{\lambda}^2}{n!} \He_{\lambda}^2(x)
   = \He^*_n(x^2)= i^{-n} \He_n(ix^2).
\end{equation*}

\subsection{Wronskians of Laguerre polynomials (P-III and P-V)}
The classical Laguerre polynomials $(L_n^{(\alpha)})_{n=0}^{\infty}$ with parameter $\alpha\in\mathbb{R}$ satisfy the differential relation $\bigl(L_n^{(\alpha)}\bigr)' = - L_{n-1}^{(\alpha+1)}$ and have constant terms
$L_n^{(\alpha)}(0)=\binom{\alpha+n}{n}$ (see \cite{Szego}),
so the \textbf{modified Laguerre polynomials}
\begin{equation*}
	l^{(\alpha)}_n(x)
		:= n! \cdot L_n^{(\alpha-n)}(-x) \qquad (n\geq 0),
\end{equation*}
form an Appell sequence with generating series $f_l(t)=(1+t)^{\alpha}$.
The constant terms are $z_k=\alpha(\alpha-1)\cdots(\alpha-k+1)$ for $k\geq 0$, and $c_k = (-1)^{k-1} (k-1)!\,\alpha$ for $k\geq 1$.
Wronskians of (modified) Laguerre polynomials, corresponding to specific partitions, are used in the rational solutions of the Painlev\'e III and Painlev\'e V equations, see \cite[6.1.2 and 6.1.4]{VanAssche}.
Laguerre polynomials have integer coefficients when $\alpha$ is an integer, and this extends to Wronskians of Laguerre polynomials by Theorem \ref{thm:integercoefficients}.

The dual Laguerre polynomials have generating function $f^*_l(t)=(1-t)^{-\alpha}$ and hence
\begin{equation*}
	\bigl(l^{(\alpha)}_n\bigr)^*(x)
		= (-1)^n \cdot l^{(-\alpha)}_n(-x).
\end{equation*}
Theorem \ref{thm:dualWronskian} therefore implies
\begin{equation*}
	l^{(\alpha)}_{\lambda'}(x)
		= (-1)^{|\lambda|} \cdot l^{(-\alpha)}_{\lambda}(-x).
\end{equation*}

By Theorem \ref{thm:average}, the average over the Plancherel measure
is $(x+\alpha)^n$. The second moment may be expressed in terms of the
Appell sequence $B=(B_n)_{n=0}^{\infty}$ determined by
\[
\log f_B(t) = \sum_{k=2}^{\infty}\alpha^2\frac{t^k}{k}
  = -\alpha^2(t+\log(1-t)).
\]
It follows that
\[
\sum_{n=0}^\infty B_n(x)\frac{t^n}{n!}
  =\frac{e^{(x-\alpha^2)t}}{(1-t)^{\alpha^2}}.
\]
This coincides with centralizing the dual Laguerre Appell sequence with
parameter $\alpha^2$; i.e.,
\[
B_n(x)=\bigl(l_n^{(\alpha^2)}\bigr)^*(x-\alpha^2)
  =(-1)^n\cdot l_n^{(-\alpha^2)}(-x+\alpha^2),
\]
and thus Theorem~\ref{thm:secondmoment} implies
\[
\mathbb{E}\Bigl(\Bigl(l^{(\alpha)}_\lambda(x)\Bigr)^2:\lambda\vdash n\Bigr)
  = B_n((x+\alpha)^2) 
  = (-1)^n\cdot l_n^{(-\alpha^2)}\bigl(-x^2-2\alpha x\bigr).
\]

The generating recurrence relation~\eqref{eq:generatingrecursion} for $\lambda\vdash n$ takes the form
\begin{equation*}
	F_\lambda l^{(\alpha)}_\lambda
		= x \sum_{\mu \covby \lambda} F_\mu l^{(\alpha)}_\mu + \alpha \sum_{k=1}^n (-1)^{k-1}\frac{(n-1)!}{(n-k)!} \sum_{\nu\in\rhdown{k}\lambda} (-1)^{\htt(\lambda/\nu)} F_\nu l^{(\alpha)}_\nu.
\end{equation*}
We note that if $\lambda$ is a partition whose Wronskian polynomial appears in the rational solution of the Painlev\'e III or V equation, the polynomials on the right hand side of this generating recurrence relation do not necessarily appear in such rational solutions as well. This is analogous to the remarks made in \cite{Bonneux_Stevens} about Wronskian Hermite polynomials and their appearance in the rational solutions of the Painlev\'e IV equation.

\subsection{Wronskians of Jacobi polynomials (P-VI)}
Rational solutions of the Painlev\'e VI equation are expressible in terms of Jacobi polynomials.
A classical formula for these polynomials with parameters $\alpha,\beta\in\mathbb{R}$ is 
\begin{equation*}
	P_n^{(\alpha,\beta)}(x) = \frac{1}{n!} \sum_{k=0}^n \binom{n}{k} (n+\alpha+\beta+1)_k (\alpha+k+1)_{n-k} \left(\frac{x-1}{2}\right)^k \qquad (n \geq 0),
\end{equation*}
where $(a)_k=a(a+1)\cdots(a+k-1)$, see \cite[Eq 4.21.2]{Szego}.
The Jacobi polynomials which come into play in the rational solutions of P-VI involve parameters that shift with $n$; i.e.,
\begin{equation*}
	P_n^{(\alpha-n,\beta-n)}(x)  
		= \frac{1}{n!} \sum_{m=0}^n \binom{n}{m} (\alpha+\beta-n+1)_{n-m} (\alpha-m+1)_{m} \left(\frac{x-1}{2}\right)^{n-m}.
\end{equation*}
When $\alpha+\beta\notin\mathbb{N}$, these polynomials have degree $n$
and the rescalings
\[
\tilde{P}_n^{(\alpha-n,\beta-n)} := \frac{2^n n!}{(\alpha+\beta-n+1)_n} P_n^{(\alpha-n,\beta-n)}
\]
are monic; we call these \textbf{modified Jacobi polynomials}.
Substituting $x \mapsto x+1$ yields
\begin{equation*}
  A_n^{(\alpha,\beta)}(x):=\tilde{P}_n^{(\alpha-n,\beta-n)}(x+1)
    = \sum_{m=0}^n \binom{n}{m} \frac{(\alpha-m+1)_m}{(\alpha+\beta-m+1)_m} 2^m x^{n-m}\qquad(n\geq0),
\end{equation*}
which by~\eqref{eq:explicitAppell} is an Appell sequence with constant terms
\begin{equation*}
	z_{k}
		= 2^k \frac{(\alpha-k+1)_k}{(\alpha+\beta-k+1)_k}
		= 2^k \frac{(-\alpha)_k}{(-\alpha-\beta)_k}.
\end{equation*}
The exponential generating function is therefore a hypergeometric function; namely,
\begin{equation*}
	f_{A^{(\alpha,\beta)}}(t) 
		= {}_1 F_1(-\alpha; -\alpha-\beta; 2t).
\end{equation*} 
Since the Appell property is preserved under translations, the modified Jacobi polynomials are also an Appell sequence.
The generating function is determined by the above formulas, but we lack
explicit formulas for the constants $c_k$. Without such formulas, the specialization of our main results to these polynomials cannot be made explicit.

\section*{Acknowledgements}
The first and last authors are supported in part by the long term structural funding-Methusalem grant of the Flemish Government, and by EOS project 30889451 of the Flemish Science Foundation (FWO). Marco Stevens is also supported by the Belgian Interuniversity Attraction Pole P07/18, and by FWO research grant G.0864.16.
Additionally, we would like to thank Guilherme Silva for introducing the second and last authors to each other.

\end{document}